\documentclass[a4paper,openany,12pt]{article} 
\usepackage{amsmath,amsfonts,amssymb,amsthm}
\usepackage{amsmath,amssymb,fancyhdr, mathrsfs, graphicx, epsfig}
\usepackage{enumerate}
\newcommand{\bc}{\begin{center}}
\newcommand{\ec}{\end{center}}
\newcommand{\ba}{\begin{array}}
\newcommand{\ea}{\end{array}}

\renewcommand{\d}{\delta }

\renewcommand{\l}{\lambda }

\renewcommand{\dfrac}{\displaystyle\frac }

\newcommand{\dint }{\displaystyle\int }

\newtheorem{thm}{Theorem\ }[section]

\newtheorem{rem}{ Remark\ }[section]
\newtheorem{cor}{Corollary }[section]

\newtheorem{lem}{ Lemma\ }[section]

\newtheorem{proposition}{\quad\quad Proposition\ }[section]

\numberwithin{equation}{section}
\numberwithin{thm}{section}
\numberwithin{defn}{section}
\numberwithin{lem}{section}
\numberwithin{cor}{section}
\numberwithin{rem}{section}

\title{Regularity Criterion to the axially symmetric   Navier-Stokes Equations\thanks{ E-mail address: jnwdyi@163.com (D.Y. Wei).} \thanks
{}}
\author{\small\bf Τ¶«ÞÈ Dongyi Wei    \\
    School of Math Sciences and BICMR,  Peking University, \\  Beijing,  ~100871,~ \  People's
 Republic  of  China.
    }
\date{}
\begin{document}
\maketitle

\setcounter{tocdepth}{3} 


\begin{abstract}
\vskip 3mm
 Smooth solutions to the axially symmetric Navier-Stokes equations obey the following
maximum principle:$\|ru_\theta(r,z,t)\|_{L^\infty}\leq\|ru_\theta(r,z,0)\|_{L^\infty}.$
We first
prove the global regularity of solutions if
$\|ru_\theta(r,z,0)\|_{L^\infty}$ or $ \|ru_\theta(r,z,t)\|_{L^\infty(r\leq r_0)}$ is
small compared with certain dimensionless quantity of the initial data. This result improves the one in Zhen Lei and Qi S. Zhang \cite{1}. As a corollary, we also
prove the global regularity under the assumption that $|ru_\theta(r,z,t)|\leq\ |\ln r|^{-3/2},\ \ \forall\ 0<r\leq\delta_0\in(0,1/2).$
\end{abstract}
\bc
\begin{minipage}{13 cm}
{\bf Key words }\ axially symmetric,\ Navier-Stokes Equations,\ Regularity Criterion. \ \vskip 0.1cm {\bf
 MSC2010} \ \ 35Q30,\
 \ 76D05,   \ 76D07, \ 76N10. \
\end{minipage}
\ec

\baselineskip 18pt

\section{Introduction}

\ \ \ \ In the cylindrical coordinate system with $(x_1,x_2,x_3)=(r\cos\theta,r\sin\theta,z)$, an axially symmetric solution of the Navier-Stokes equations is a solution of the following form
$$u(x,t)=u_r(r,z,t)e_r+u_{\theta}(r,z,t)e_{\theta}+u_z(r,z,t)e_z,\ p(x,t)=p(r,z,t),$$ where $$e_r=\left(\frac{x_1}{r},\frac{x_2}{r},0\right),\ e_{\theta}=\left(-\frac{x_2}{r},\frac{x_1}{r},0\right),\ e_z=(0,0,1).$$
In terms of $(u_r, u_{\theta}, u_z, p)$, the axially symmetric Navier-Stokes equations are as follows
\begin{equation}\left\{\ba{l}\partial_t u_r+u\cdot\nabla u_r-\triangle u_r+\frac{u_r}{r^2}-\frac{u_{\theta}^2}{r}+\partial_r p=0,\\[3mm] \partial_t u_{\theta}+u\cdot\nabla u_{\theta}-\triangle u_{\theta}+\frac{u_{\theta}}{r^2}+\frac{u_ru_{\theta}}{r}=0,\\[3mm]
\partial_t u_z+u\cdot\nabla u_z-\triangle u_z+\partial_z p=0,\\[3mm]
\partial_r (ru_r)+\partial_z (ru_z)=0.\ea\right.\end{equation}
It is well-known that finite energy smooth solutions of the Navier-Stokes equations satisfy
the following energy identity \begin{equation}\label{nseq1.2}
   \|u(t)\|_{L^2}^2+2\int_0^t\|\nabla u(s)\|_{L^2}^2\mathrm{d}s=\|u_0\|_{L^2}^2<+\infty.
 \end{equation}
 Denote $\Gamma=ru_\theta$. One can easily check that \begin{equation}
   \partial_t \Gamma+u\cdot\nabla \Gamma-\triangle \Gamma+\frac{2}{r}\partial_r \Gamma=0.
 \end{equation}
 A significant consequence of (1.3) is that smooth solutions of the axially symmetric Navier-Stokes
equations satisfy the following maximum principle (see, for instance,\cite{2}\cite{3})\begin{equation}
   \|\Gamma\|_{L^\infty}\leq\|\Gamma_0\|_{L^\infty}.
 \end{equation}
 We can compute the vorticity $$\omega=\nabla\times u=\omega_r e_r+\omega_{\theta} e_{\theta}+\omega_z e_z,$$
where\ \ $$\omega_r=-\partial_z(u_{\theta}),\\ \omega_{\theta}=\partial_z(u_r)-\partial_r(u_z),\ \omega_z=\dfrac{1}{r}\partial_r(ru_{\theta}). $$ Denote $$\Omega=\dfrac{\omega_{\theta}}{r},\ J=\dfrac{\omega_{r}}{r}=-\dfrac{\partial_z u_\theta}{r},$$ then
\begin{equation}\left\{\ba{l}\partial_t \Omega+u\cdot\nabla \Omega-\left(\triangle+\dfrac{2}{r}\partial_r\right)\Omega+2\dfrac{u_{\theta}}{r}J=0,\\[2mm] \partial_t J+u\cdot\nabla J-\left(\triangle+\dfrac{2}{r}\partial_r\right)J-\left(\omega_r \partial_r+\omega_z \partial_z \right)\dfrac{u_r}{r}=0.\ea\right.\end{equation}
We emphasis that $J$ was introduced by Chen-Fang-Zhang in \cite{5}, while
$\Omega$ appeared much
earlier and can be at least tracked back to the book of Majda-Bertozzi in \cite{17}. Both of the
two new variables are of great importance in our work.

 Our goal is to prove that the smallness of $\|\Gamma\|_{L^\infty(r\leq r_0)}$ or $\|\Gamma_0\|_{L^\infty}$
implies the global regularity of the solutions.  Here is our result.
\begin{thm}\label{thm1.1}
    Let $r_0 > 0$. Suppose that $u_0\in H^2$ such that
 $\Gamma_0\in L^\infty$. Denote
 $$M_1=(1+\|\Gamma_0\|_{L^\infty})\|u_0\|_{L^2} \ \mbox{and} \  M_0=(\|J_0\|_{L^2}+\|\Omega_0\|_{L^2})M_1^3.$$ Then there exists an absolute positive constant $C_0 > 0$ such that if  $$\ \ (a)\ \ \|\Gamma\|_{L^\infty(r\leq r_0)}\leq \left(1+\ln\left(C_0\max\left\{M_0^{1/4},r_0^{-1/2}{M_1}\right\}
 +1\right)\right)^{-3/2},$$ then the axially symmetric Navier-Stokes equations are globally well-posed.
\end{thm}
\begin{rem}\label{rem1.1} Choose $r_0 > 0$ such that $ M_0^{1/4}\geq r_0^{-1/2}{M_1}$ and use (1.4), then we can obtain the following global regularity condition
 $$(b)\ \ \|\Gamma_0\|_{L^\infty}\leq (1+\ln(C_0M_0^{1/4}+1))^{-3/2},\ \ $$ this condition depends only on the initial value and is very useful especially when $\|\Gamma_0\|_{L^\infty}$ is very small, in this sense it improves the result in \cite{1}. On the other hand, if we take $r_0\rightarrow0^+$ in condition
$(a)$ we can obtain an important corollary.
\end{rem}

\begin{cor}\label{cor1.1}Let $\delta_0\in (0, 1/
2 ),\  u$ be the strong solution of the axially
symmetric Navier-Stokes equations with initial value $u_0\in H^2$ and $\|\Gamma_0\|_{L^\infty}<\infty$. If\begin{equation}\label{1.6}
   |\Gamma(r,z,t)|\leq\ |\ln r|^{-3/2},\ \ \forall\ 0<r\leq\delta_0,
 \end{equation}then $u$ is regular globally in time.
\end{cor}
 Denote $$K(\varepsilon)=\exp\left(\sqrt{2\varepsilon^{-\frac{4}{3}}-1}-1\right),\ \ K_0(\varepsilon)=\exp(\varepsilon^{-\frac{2}{3}}-1)$$ for $0<\varepsilon\leq1$, then one can easily check that $$1+\ln K(\varepsilon)+\dfrac{1}{2}(\ln K(\varepsilon))^2=\varepsilon^{-\frac{4}{3}},\ \ \varepsilon^{\frac{4}{3}} K(\varepsilon)\geq \frac{K_0(\varepsilon)}{C_*}>0,$$ for some absolute
positive constant $C_*$ and $0<\varepsilon\leq1$. The use of the functions $K$ and $K_0$ is due to a new important observation in Lemma 2.3.

Throughout this paper, we assume $u\in C([0,T^*);H^2)$ to be the unique strong solution to the Navier-Stokes equations $(1.1)$ with initial value $u(0)=u_0$, and the maximal existence time
 $T^*>0$. We also assume $u$ to be axially
symmetric with $ u_r, u_{\theta}, u_z, \Gamma, \Omega, J$ defined above, and denote $$ \Gamma_0=\Gamma(0),\ \Omega_0=\Omega(0),\ J_0=J(0),\ \|\Gamma\|_{L^\infty}=\|\Gamma\|_{L^\infty(\mathbb{R}^3\times(0,T^*))}.$$ Due to the regularity of solutions to Navier-Stokes equations, $u\in C((0,T^*);H^4),\  $ and $$ \Omega,\ J\in C([0,T^*);L^{2})\cap C((0,T^*);H^{2}),$$  $$\dfrac{ u_{\theta}}{r},\ \dfrac{ u_{r}}{r}\in C((0,T^*);H^{3}), \ \partial_r
\dfrac{u_r}{r}\bigg|_{r=0}=0. $$ Hence, all calculations below are legal for $t\in (0,T^*)$. (see \cite{1} for more explaination)

Now let us recall some highlights on the study of the axially symmetric Navier-Stokes equations. If the  swirl
$u_{\theta} =0$, global regularity  result was  proved independently by Ukhovskii and Yudovich \cite{39},
and Ladyzhenskaya
\cite{24}, also \cite{26} for a refined proof. In the presence of swirl, the global regularity problem is still
open. Recently, tremendous efforts and interesting progress have been made on the regularity problem of the axially symmetric Navier-Stokes equations \cite{2}\cite{3}\cite{10}\cite{5}\cite{8}\cite{9}\cite{1}. There are many significant results under the sufficient conditions for regularity of axially symmetric solution of type $$\omega_{\theta}\in L^p(0,T;L^q(\mathbb{R}^3))\ \ \mbox{and}\ \
\dfrac{u_r}{r}\in L^p(0,T;L^q(\mathbb{R}^3)),\ \dfrac{2}{p}+\dfrac{3}{q}\leq2,\ \dfrac{3}{2}<q<+\infty
$$ in \cite{2}.
 It has been shown in \cite{5} that the axially symmetric solution is smooth in $\mathbb{R}^3\times(0,T]$ when $r^du_{\theta}\in L^p(0,T;L^q(\mathbb{R}^3))$ with
 $$
 d\in[0,1),\ (p,q)\in\left\{\left[\frac{2}{1-d},\infty\right]\times\left(\frac{3}{1-d},\infty\right],\dfrac{2}{p}+\dfrac{3}{q}\leq1-d\right\},
 $$
  in particular, global regularity is obtained if $|\Gamma|\leq Cr^{\alpha}$ for some $\alpha>0,\ C>0$. In \cite{1} global regularity is obtained if $|\Gamma|\leq C|\ln r|^{-2}$ for some $C>0$. Clearly, our Corollary 1.1 improves the one in \cite{1}.

Here global regularity means $T^*=+\infty$, and we only need to prove $\Omega\in L^\infty(0,T^*;L^2(\mathbb{R}^3)),$ hence we can use the results in \cite{5} or \cite{20} to obtain global regularity. We can also use Lemma 2.1 in section 2 to obtain $ \nabla \dfrac{ u_{r}}{r}\in\ L^\infty(0,T^*;L^2(\mathbb{R}^3)),$  and $\dfrac{ u_{r}}{r} \in\ L^\infty(0,T^*;L^6(\mathbb{R}^3)) ,$ then the results in \cite{2} or \cite{23} imply the global regularity.

This paper is organized as follows, in section 2 we will give some notations and 3 important Lemmas, in section 3 we will first follow the proof in \cite{1} then use the Lemmas in section 2 to conclude the proof.
\section{ Notations and Lemmas}

The Laplacian operator $\triangle$ and the
gradient operator $\nabla$ in the cylindrical coordinate are
$$\triangle=\partial_r^2+\frac{1}{r}\partial_r+\frac{1}{r^2}\partial_\theta
+\partial_z^2,\ \ \nabla=e_r\partial_r+\frac{e_\theta}{r}\partial_\theta+e_z\partial_z.$$
We will use $C$ to denote a generic absolute
positive constant whose meaning may change from line to line.
 If $|f|^2$ is axially
symmetric, we will denote
 $$\|f\|_{L^2}^2=\int|f|^2r\mathrm{d}r\mathrm{d}z,\ \ \mathrm{d}x=r\mathrm{d}r\mathrm{d}z.$$
 The following estimate will be used very often. (see \cite{8}\cite{12}\cite{19})
 \begin{lem}\label{lem2.1}
   $\left\|\nabla\dfrac{u_r}{r}\right\|_{L^2}^2\leq\|\Omega\|_{L^2}^2,\ \left\|\nabla^2\dfrac{u_r}{r}\right\|_{L^2}^2\leq\left\|\partial_z\Omega\right\|_{L^2}^2.$
\end{lem}
\begin{proof}[\bf Proof.]\  By virtue of $\partial_r (ru_r)+\partial_z (ru_z)=0$, we can find the stream function $\psi_{\theta}$ such that $$u_r=-\partial_z\psi_{\theta},\ u_z=\dfrac{1}{r}\partial_r(r\psi_{\theta}),$$ then we can compute $$-\left(\triangle+\dfrac{2}{r}\partial_r\right)\dfrac{\psi_{\theta}}{r}=\Omega,\ \left(\triangle+\dfrac{2}{r}\partial_r\right)\dfrac{u_r}{r}=\partial_z\Omega.$$
 Using integration by parts, we have
  $$
  -\dint \dfrac{u_r}{r}\left(\triangle+\frac{2}{r}\partial_r\right)\dfrac{u_r}{r}\mathrm{d}x
=\left\|\nabla \dfrac{u_r}{r}\right\|_{L^2}^2+\int\left|\dfrac{u_r}{r}(0,z,t)\right|^2\mathrm{d}z,
$$ therefore,
 $$\ba{c}\left\|\nabla \dfrac{u_r}{r}\right\|_{L^2}^2\leq-\dint \dfrac{u_r}{r}\left(\triangle+\frac{2}{r}\partial_r\right)\dfrac{u_r}{r}\mathrm{d}x
\\[3mm]=-\dint \dfrac{u_r}{r}\partial_z\Omega\mathrm{d}x=\dint \partial_z\dfrac{u_r}{r}\Omega\mathrm{d}x\leq\left\|\nabla \dfrac{u_r}{r}\right\|_{L^2}\left\|\Omega\right\|_{L^2},\ea $$
 hence,
$$ \left\|\nabla \dfrac{u_r}{r}\right\|_{L^2}\leq\|\Omega\|_{L^2}.$$ Since
$$
 \ba{c}\left\|\partial_z\Omega\right\|_{L^2}^2=\left\|\left(\triangle+\dfrac{2}{r}\partial_r\right)
\dfrac{u_r}{r}\right\|_{L^2}^2\\[4mm]=\left\|\triangle
\dfrac{u_r}{r}\right\|_{L^2}^2+4\left\|\dfrac{1}{r}\partial_r
\dfrac{u_r}{r}\right\|_{L^2}^2+4\dint\triangle
\dfrac{u_r}{r}\dfrac{1}{r}\partial_r
\dfrac{u_r}{r}r\mathrm{d}r\mathrm{d}z,\ea
$$ and
$$\left\|\triangle
\dfrac{u_r}{r}\right\|_{L^2}^2=\left\|\nabla^2
\dfrac{u_r}{r}\right\|_{L^2}^2,\ \ \ \partial_r
\dfrac{u_r}{r}\bigg|_{r=0}=0,$$
$$\int\triangle
\dfrac{u_r}{r}\dfrac{1}{r}\partial_r
\dfrac{u_r}{r}r\mathrm{d}r\mathrm{d}z=\left\|\dfrac{1}{r}\partial_r
\dfrac{u_r}{r}\right\|_{L^2}^2+\frac{1}{2}\int\left(\left|\partial_z
\dfrac{u_r}{r}\right|^2-\left|\partial_r
\dfrac{u_r}{r}\right|^2\right)(0,z,t)\mathrm{d}z\geq 0,
$$
 we can obtain $ \ \left\|\nabla^2
\dfrac{u_r}{r}\right\|_{L^2}^2\leq\left\|\partial_z\Omega\right\|_{L^2}^2,$
this completes the proof Lemma \ref{lem2.1}.\end{proof}

Denote $$v(r,z,t)=\dint_0^r|u_\theta(r',z,t)|\mathrm{d}r',\ \mathrm{ for}  \ r>0,\ \ a(t)=\left\|\dfrac{v}{r}(t)\right\|_{L^\infty},$$ then we have the following inequality
\begin{lem}\label{lem2.2}
   $a(t)^2\leq\|J(t)\|_{L^2}\left\|\dfrac{u_\theta}{r}(t)\right\|_{L^2}$.
\end{lem}
\begin{proof}[\bf Proof.]\ For $r'>0,\ z'\in\mathbb{R}, \ t>0$ as $ v(r',z',t)=\dint_0^{r'}|u_\theta(r,z',t)|\mathrm{d}r,$ by $\mathrm{H\ddot{o}lder}$ inequality  we have
$$\ba{ll} |v(r',z',t)|^2&\leq\dint_0^{r'}r\mathrm{d}r\int_0^{r'}\dfrac{|u_\theta(r,z',t)|^2}
{r}\mathrm{d}r
\\[3mm]&=\dfrac{r'^2}{2}\int_0^{r'}\mathrm{d}r\int_z^{+\infty}-\partial_z
\dfrac{|u_\theta(r,z,t)|^2}
{r}\mathrm{d}z\\[3mm]&
=\dfrac{r'^2}{2}\int_0^{r'}\mathrm{d}r\int_z^{+\infty}
2Ju_\theta(r,z,t)\mathrm{d}z\\[3mm]&
\leq r'^2\dint|J|\left|\dfrac{u_\theta}
{r}\right|(r,z,t)r\mathrm{d}r\mathrm{d}z
\\[3mm]& \leq r'^2\|J(t)\|_{L^2}\left\|\dfrac{u_\theta}{r}(t)\right\|_{L^2}.
\ea$$
Hence, we get $$\ a(t)^2=\sup\limits_{r'>0,z'\in\mathbb{R}}\left|\dfrac{v(r',z',t)}{r'}\right|^2\leq
\|J(t)\|_{L^2}\left\|\dfrac{u_\theta}{r}(t)\right\|_{L^2}.$$ This completes the proof.\end{proof}

\begin{lem}\label{lem2.3} Assume that $t>0,\ \|\Gamma\|_{L^\infty(r\leq r_1)}\leq \varepsilon\leq1,$ and $\ 0<r_1\leq\dfrac{\varepsilon K(\varepsilon)}{a(t)}$, then
   \begin{equation}\label{2.3}
   \int\dfrac{|u_\theta(t)|}{r}|f|^2\mathrm{d}x\leq \varepsilon^{-\frac{1}{3}}\int|\partial_rf|^2\mathrm{d}x
   +C\frac{\|\Gamma\|_{L^\infty}+\varepsilon^{-\frac{1}{3}}}{r_1^2}
   \int_{r\geq\frac{r_1}{2}}|f|^2\mathrm{d}x,\end{equation}
   \begin{equation}\label{2.4}
   \int |u_\theta(t)|^2|f|^2\mathrm{d}x\leq \varepsilon^{\frac{2}{3}}\int|\partial_rf|^2\mathrm{d}x
   +C\frac{\|\Gamma\|_{L^\infty}^2+\varepsilon^{\frac{2}{3}}}{r_1^2}
   \int_{r\geq\frac{r_1}{2}}|f|^2\mathrm{d}x,\end{equation}
 for all axially symmetric scalar and vector
functions $f\in H^1$.\end{lem}
\begin{proof}[\bf Proof.]\ We first prove that if $f=0$ for $r\geq r_1$, then
\begin{equation}\label{2.5}
    \int\dfrac{|u_\theta(t)|}{r}|f|^2r\mathrm{d}r\mathrm{d}z\leq \varepsilon^{-\frac{1}{3}}\int|\partial_rf|^2r\mathrm{d}r\mathrm{d}z.
\end{equation}
In this case, $f(r',z)=-\dint_{r'}^{r_1}\partial_rf(r,z)\mathrm{d}r$ for $0<r'<r_1,$ by $\mathrm{H\ddot{o}lder}$ inequality  we have
$$|f(r',z)|^2\leq\dint_{r'}^{r_1}r|\partial_rf(r,z)|^2\mathrm{d}r
\int_{r'}^{r_1}\frac{\mathrm{d}r}{r}.
$$
Consequently,
\begin{equation}\label{2.6}
    \ \int|u_\theta(r',z,t)||f(r',z)|^2\mathrm{d}r'
    \leq\int r|\partial_rf(r,z)|^2\mathrm{d}r\int_0^{r_1}|u_\theta(r',z,t)|
    \int_{r'}^{r_1}\frac{\mathrm{d}r}{r}\mathrm{d}r',
\end{equation}
by the definition of $v$ we have \begin{equation}\label{2.7}
    \int_0^{r_1}|u_\theta(r',z,t)|
    \int_{r'}^{r_1}\frac{\mathrm{d}r}{r}\mathrm{d}r'=
    \int_0^{r_1}v(r,z,t)\frac{\mathrm{d}r}{r},
\end{equation}
by the definition of $a(t)$ we have $v(r,z,t)\leq ra(t).$ On the other hand, $|u_\theta|=\dfrac{|\Gamma|}{r}\leq \dfrac{\varepsilon}{r}$ for $0<r<r_1.$ Hence, if $\dfrac{\varepsilon}{a(t)}\leq r\leq r_1$, then
$$\ba{l}
v(r,z,t)=v\left(\dfrac{\varepsilon}{a(t)},z,t\right)+\dint_{\frac{\varepsilon}{a(t)}}^r
|u_\theta(r',z,t)|\mathrm{d}r'\\[4mm]\leq\dfrac{\varepsilon}{a(t)}a(t)+
\int_{\frac{\varepsilon}{a(t)}}^r\dfrac{\varepsilon}{r'}\mathrm{d}r'=
\varepsilon\left(1+\ln\frac{ra(t)}{\varepsilon}\right).\ea
$$ The above estimates of $v$ implies\begin{equation}\label{2.8}
    \ba{l}
    \dint_0^{r_1}v(r,z,t)\dfrac{\mathrm{d}r}{r}\leq
    \int_0^{\frac{\varepsilon}{a(t)}}\dfrac{ra(t)}{r}\mathrm{d}r+
    \int_{\frac{\varepsilon}{a(t)}}^{\frac{\varepsilon K(\varepsilon)}{a(t)}}\varepsilon\left(1+\ln\dfrac{ra(t)}{\varepsilon}\right)
    \frac{\mathrm{d}r}{r}\\=\varepsilon+\dint_1^{K(\varepsilon)}
    \varepsilon(1+\ln r)\frac{\mathrm{d}r}{r}=\varepsilon\left(1+\ln K(\varepsilon)+\dfrac{1}{2}\left(\ln K(\varepsilon)\right)^2\right)=\varepsilon^{-\frac{1}{3}},\ea
\end{equation} here we used $0<r_1\leq\dfrac{\varepsilon K(\varepsilon)}{a(t)}.$  By $\eqref{2.6}$, $\eqref{2.7}$, $\eqref{2.8}$, we have
$$\int|u_\theta(r',z,t)||f(r',z)|^2\mathrm{d}r'
    \leq\varepsilon^{-\frac{1}{3}}\int r|\partial_rf|^2\mathrm{d}r,$$
    integrate in $z$, we obtain $\eqref{2.5}$. Now we discuss general $f$. Take a smooth cut-off function of $r$ such that $(i)\ \ \phi'\leq 0,\ \ (ii)\ \ \phi\equiv1$ if $0\leq r\leq\dfrac{1}{2}$, $\ \ (iii)\ \ \phi\equiv0$ if $ r\geq1$. Then we have
    $$
    \ba{l}\ \ \dint\dfrac{|u_\theta(t)|}{r}|f|^2r\mathrm{d}r\mathrm{d}z=
    \int|u_\theta(t)|\left|\phi\left(\frac{r}{r_1}\right)f\right|^2
    \mathrm{d}r\mathrm{d}z+\int|u_\theta(t)|\left(1-\phi\left(\frac{r}{r_1}\right)^2\right)|f|^2
    \mathrm{d}r\mathrm{d}z\\[3mm]
    \leq\varepsilon^{-\frac{1}{3}}
    \dint\left|\partial_r\left[\phi\left(\frac{r}{r_1}\right)f\right]\right|^2r\mathrm{d}r\mathrm{d}z+
    \int_{r\geq\frac{r_1}{2}}\frac{|\Gamma|}{r}|f|^2\mathrm{d}r\mathrm{d}z\\[2mm]
    \leq\varepsilon^{-\frac{1}{3}}\left(\dint
    |\partial_rf|^2r\mathrm{d}r\mathrm{d}z+\frac{C}{r_1^2}\int_{r\geq\frac{r_1}{2}}
    |f|^2r\mathrm{d}r\mathrm{d}z\right)+\dfrac{4\|\Gamma\|_{L^\infty}}
    {r_1^2}\dint_{r\geq\frac{r_1}{2}}
    |f|^2r\mathrm{d}r\mathrm{d}z\\[2mm]\leq\varepsilon^{-\frac{1}{3}}
    \dint|\partial_rf|^2r\mathrm{d}r\mathrm{d}z
   +C\frac{\|\Gamma\|_{L^\infty}+\varepsilon^{-\frac{1}{3}}}{r_1^2}
   \int_{r\geq\frac{r_1}{2}}|f|^2r\mathrm{d}r\mathrm{d}z,\ea
   $$
   here we used $\eqref{2.5}$ and the fact that
   $$
   \ba{l}\dint\left|\partial_r\left[\phi\left(\frac{r}{r_1}\right)f\right]\right|^2r\mathrm{d}r\mathrm{d}z
   =\int\Bigg[\phi\left(\frac{r}{r_1}\right)^2|\partial_rf|^2+|f|^2
   \left|\partial_r\phi\left(\frac{r}{r_1}\right)\right|^2\\
  \hskip 4cm +\partial_r|f|^2
   \phi\left(\dfrac{r}{r_1}\right)\partial_r\phi\left(\dfrac{r}{r_1}\right)\Bigg]r\mathrm{d}r\mathrm{d}z\\[3mm]
   =\dint\left[\phi\left(\frac{r}{r_1}\right)^2|\partial_rf|^2+|f|^2
   \left|\partial_r\phi\left(\frac{r}{r_1}\right)\right|^2\right]r\mathrm{d}r\mathrm{d}z-\int|f|^2
   \partial_r\left[\phi\left(\frac{r}{r_1}\right)\partial_r\phi\left(\frac{r}{r_1}\right)r\right]
   \mathrm{d}r\mathrm{d}z\\
   \leq\dint|\partial_rf|^2+\frac{C}{r_1^2}\int_{r\geq\frac{r_1}{2}}
    |f|^2r\mathrm{d}r\mathrm{d}z.\ea
    $$
    Similarly we have
    $$
    \ba{l}\ \ \dint|u_\theta(t)|^2|f|^2r\mathrm{d}r\mathrm{d}z=
    \int|u_\theta(t)|^2\left|\phi\left(\frac{r}{r_1}\right)f\right|^2r
    \mathrm{d}r\mathrm{d}z
    +\int|u_\theta(t)|^2\left(1-\phi\left(\frac{r}{r_1}\right)^2\right)|f|^2
    r\mathrm{d}r\mathrm{d}z\\[3mm]
    \leq\|\Gamma\|_{L^\infty(r\leq r_1)}\dint|u_\theta(t)|\left|\phi\left(\frac{r}{r_1}\right)f\right|^2
    \mathrm{d}r\mathrm{d}z+
    \int_{r\geq\frac{r_1}{2}}\frac{|\Gamma|^2}{r^2}|f|^2
    r\mathrm{d}r\mathrm{d}z\\[2mm]
    \leq\varepsilon\varepsilon^{-\frac{1}{3}}\left(\dint
    |\partial_rf|^2r\mathrm{d}r\mathrm{d}z+
    \frac{C}{r_1^2}\int_{r\geq\frac{r_1}{2}}
    |f|^2r\mathrm{d}r\mathrm{d}z\right)+\dfrac{4\|\Gamma\|_{L^\infty}^2}
    {r_1^2}\int_{r\geq\frac{r_1}{2}}
    |f|^2r\mathrm{d}r\mathrm{d}z\\[2mm]
    \leq\varepsilon^{\frac{2}{3}}
    \dint|\partial_rf|^2r\mathrm{d}r\mathrm{d}z
   +C\frac{\|\Gamma\|_{L^\infty}^2+\varepsilon^{\frac{2}{3}}}{r_1^2}
   \int_{r\geq\frac{r_1}{2}}|f|^2r\mathrm{d}r\mathrm{d}z.\ea
   $$
 This completes the proof. \end{proof}

\section{  Proof of the results}

\begin{proof}[\bf Proof of  Theorem\ 1.1.]\  By applying standard energy estimate to $J$ equation, we have $$\frac{1}{2}\frac{\mathrm{d}}{\mathrm{d}t}\|J\|_{L^2}^2+\int J(u\cdot\nabla) J\mathrm{d}x-\int J\left(\triangle+\frac{2}{r}\partial_r\right)J\mathrm{d}x-\int J(\omega_r \partial_r+\omega_z \partial_z )\frac{u_r}{r}\mathrm{d}x=0.$$
Using $\nabla\cdot u=0,$ one has
$$\dint J(u\cdot\nabla) J\mathrm{d}x=\frac{1}{2}
\int J^2(\nabla\cdot u) \mathrm{d}x=0.$$ On the other hand, by direct calculations, one has $$-\dint J\left(\triangle+\frac{2}{r}\partial_r\right)J\mathrm{d}x
=\|\nabla J\|_{L^2}^2+\int|J(0,z,t)|^2\mathrm{d}z.$$
Consequently, we have
 \begin{equation}\label{3.1}
 \frac{1}{2}\frac{\mathrm{d}}{\mathrm{d}t}\|J\|_{L^2}^2+\|\nabla J\|_{L^2}^2+\int|J(0,z,t)|^2\mathrm{d}z=\int J(\omega_r \partial_r+\omega_z \partial_z )\frac{u_r}{r}\mathrm{d}x.
 \end{equation}
 Similarly, by applying the energy estimate to the equation of $\Omega$, one obtains that
\begin{equation}\label{3.2}
\frac{1}{2}\frac{\mathrm{d}}{\mathrm{d}t}\|\Omega\|_{L^2}^2
+\|\nabla \Omega\|_{L^2}^2+\int|\Omega(0,z,t)|^2\mathrm{d}z=-2\int \frac{u_\theta}{r}J\Omega\mathrm{d}x.
 \end{equation}
Notice that
$$\ba{l}\ \ \dint J(\omega_r \partial_r+\omega_z \partial_z )\frac{u_r}{r}\mathrm{d}x=\int[\nabla\times(u_{\theta}e_{\theta})]
\cdot\left(J\nabla\frac{u_\theta}{r}\right)\mathrm{d}x\\[2mm]
=\dint u_{\theta}e_{\theta}\cdot\left(\nabla J\times\nabla\frac{u_\theta}{r}\right)\mathrm{d}x
\leq\frac{1}{2}\|\nabla J\|_{L^2}^2+\frac{1}{2}\left\|u_\theta\nabla \frac{u_r}{r}\right\|_{L^2}^2,\ea
$$
and by $\eqref{3.1}$, we have
\begin{equation}\label{3.3}
 \frac{\mathrm{d}}{\mathrm{d}t}\|J\|_{L^2}^2+\|\nabla J\|_{L^2}^2+2\int|J(0,z,t)|^2\mathrm{d}z\leq\left\|u_\theta\nabla \frac{u_r}{r}\right\|_{L^2}^2.
 \end{equation}
 Now we estimate the right hand side of $\eqref{3.2}$ and $\eqref{3.3}$.
under the assumption of condition (b), let \begin{equation}\label{b}
 \varepsilon=\left(1+\ln\left(C_0\max\left\{M_0^{1/4},r_0^{-1/2}{M_1}\right\}
 +1\right)\right)^{-3/2},\end{equation} then $ \ \|\Gamma\|_{L^\infty(r\leq r_0)}\leq \varepsilon\leq 1,$  and we can apply Lemma 2.3 with $$r_1=r(t)=\min\left\{\dfrac{\varepsilon K(\varepsilon)}{a(t)},r_0\right\},$$
 and take $f=J,\ \Omega$ in $\eqref{2.3}$, we have
 \begin{equation}\label{3.4}\ba{c}
-2\dint \frac{u_\theta}{r}J\Omega\mathrm{d}x\leq\varepsilon^{\frac{1}{3}}\int \frac{|u_\theta|}{r}|\Omega|^2\mathrm{d}x+\varepsilon^{-\frac{1}{3}}\int \frac{|u_\theta|}{r}|J|^2\mathrm{d}x\\[2mm]\leq\dint
|\partial_r\Omega|^2\mathrm{d}x+\frac{C(1+\varepsilon^{\frac{1}{3}}
\|\Gamma\|_{L^\infty})}{r(t)^2}\int_{r\geq\frac{r(t)}{2}}
|\Omega|^2\mathrm{d}x\\[2mm]+\varepsilon^{-\frac{2}{3}}\dint
|\partial_rJ|^2\mathrm{d}x+\frac{C\varepsilon^{-\frac{2}{3}}(1+
\varepsilon^{\frac{1}{3}}
\|\Gamma\|_{L^\infty})}{r(t)^2}
\int_{r\geq\frac{r(t)}{2}}
|J|^2\mathrm{d}x.\ea \end{equation}
Choosing  $f=\partial_r\dfrac{u_r}{r},\ \partial_z\dfrac{u_r}{r}$ in $\eqref{2.4}$, we have
 \begin{equation}\label{3.5}
\left\|u_\theta\nabla \dfrac{u_r}{r}\right\|_{L^2}^2\leq\varepsilon^{\frac{2}{3}} \dint
\left|\partial_r\nabla\dfrac{u_r}{r}\right|^2\mathrm{d}x+
\frac{C\varepsilon^{\frac{2}{3}}(1+\varepsilon^{-\frac{2}{3}}
\|\Gamma\|_{L^\infty}^2)}{r(t)^2}\int_{r\geq\frac{r(t)}{2}}
\left|\nabla \dfrac{u_r}{r}
\right|^2\mathrm{d}x.
\end{equation}
 Denote $$M_2=1+\varepsilon^{\frac{1}{3}}\|\Gamma\|_{L^\infty}
 +\varepsilon^{-\frac{2}{3}}
\|\Gamma\|_{L^\infty}^2,$$    by Lemma 2.1 we have
$$\dint
\left|\partial_r\nabla\dfrac{u_r}{r}\right|^2\mathrm{d}x\leq\left\|\nabla^2 \dfrac{u_r}{r}\right\|_{L^2}^2\leq\left\|\partial_z \Omega\right\|_{L^2}^2.$$
Inserting $\eqref{3.4}$, $\eqref{3.5}$ into $\eqref{3.2}$, $\eqref{3.3}$,  we have
\begin{equation}\label{3.6}
 \frac{\mathrm{d}}{\mathrm{d}t}\left(\|J\|_{L^2}^2+
 \dfrac{\varepsilon^{\frac{2}{3}}}{2}\|\Omega\|_{L^2}^2\right)
 \leq  \frac{CM_2}{r(t)^2}\int_{r\geq\frac{r(t)}{2}}\left[\varepsilon^{\frac{2}{3}}
 \left(\left|\nabla \dfrac{u_r}{r}\right|^2+|\Omega|^2\right)+|J|^2\right]\mathrm{d}x.
 \end{equation}
 Denote
 $$A(t)=\|J(t)\|_{L^2}^2+
 \dfrac{\varepsilon^{\frac{2}{3}}}{2}\|\Omega(t)\|_{L^2}^2,$$ then $$A(t)\in C[0,T^*)\cap C^1(0,T^*).$$  By Lemma2.1 and the fact that
 $$\nabla\dfrac{u_r}{r}=\dfrac{\nabla u_r}{r}-\dfrac{u_r}{r^2}e_r,
 $$
   $$|\nabla u|^2=|\nabla u_r|^2+|\nabla u_{\theta}|^2
 +|\nabla u_z|^2+\left|\dfrac{u_r}{r}\right|^2+\left|\dfrac{u_{\theta}}{r}\right|^2,$$ we obtain
  $$
 \int_{r\geq\frac{r(t)}{2}}\left(\left|\nabla \dfrac{u_r}{r}\right|^2+|\Omega|^2\right)\mathrm{d}x\leq\left\|\nabla \dfrac{u_r}{r}\right\|_{L^2}^2+\|\Omega\|_{L^2}^2\leq 2\|\Omega\|_{L^2}^2,
 $$
 $$\int_{r\geq\frac{r(t)}{2}}\left(\left|\nabla \dfrac{u_r}{r}\right|^2+|\Omega|^2+|J|^2\right)\mathrm{d}x\leq \frac{C}{r(t)^2}\int |\nabla u|^2\mathrm{d}x,
 $$
hence we have
\begin{equation}\label{3.7}
\ba{l}\ \dfrac{\mathrm{d}}{\mathrm{d}t}A(t)\leq \frac{CM_2}{r(t)^2}\dint_{r\geq\frac{r(t)}{2}}
 \left[\varepsilon^{\frac{2}{3}}\left(\left|\nabla \dfrac{u_r}{r}\right|^2+|\Omega|^2\right)+|J|^2\right]\mathrm{d}x
 \\[5mm] \leq \dfrac{CM_2}{r(t)^2}\min\left\{2\varepsilon^{\frac{2}{3}}\|\Omega\|_{L^2}^2
 +\|J\|_{L^2}^2, \ \  \dfrac{C}{r(t)^2}\int |\nabla u|^2\mathrm{d}x\right\}\\[5mm] \leq \dfrac{CM_2}{r(t)^2}\min\left\{A(t), \ \ \dfrac{\|\nabla u(t)\|_{L^2}^2}{r(t)^2}\right\}.\ea
 \end{equation}
  Fix $t\in(0,T^*)$, if $r(t)=\dfrac{\varepsilon K(\varepsilon)}{a(t)}\leq r_0$, by Lemma 2.2, we have
  $$a(t)^2\leq A(t)^{\frac{1}{2}}\|\nabla u(t)\|_{L^2},$$
   $$ \dfrac{1}{r(t)^2}=\dfrac{a(t)^2}{(\varepsilon K(\varepsilon))^2}\leq\dfrac{A(t)^{\frac{1}{2}}\|\nabla u(t)\|_{L^2}}{(\varepsilon K(\varepsilon))^2}. $$ \ And  $\eqref{3.7}$ implies
 $$
 \ba{l}\ \ \ \dfrac{\mathrm{d}}{\mathrm{d}t}A(t)\leq \dfrac{CM_2A(t)^{\frac{1}{2}}\|\nabla u(t)\|_{L^2}}{(\varepsilon K(\varepsilon))^2}\min\left\{A(t), \ \ \dfrac{A(t)^{\frac{1}{2}}\|\nabla u(t)\|_{L^2}^3}{(\varepsilon K(\varepsilon))^2}\right\}\\[5mm]
 =\dfrac{CM_2A(t)\|\nabla u(t)\|_{L^2}}{(\varepsilon K(\varepsilon))^2}\min\left\{A(t)^{\frac{1}{2}}, \ \ \dfrac{\|\nabla u(t)\|_{L^2}^3}{(\varepsilon K(\varepsilon))^2}\right\}\\[5mm]\leq\dfrac{CM_2A(t)^{\frac{4}{3}}\|\nabla u(t)\|_{L^2}^2}{(\varepsilon K(\varepsilon))^{\frac{8}{3}}},\ea
 $$ otherwise, we have $r(t)=r_0$ and
 $$\dfrac{\mathrm{d}}{\mathrm{d}t}A(t)\leq \dfrac{CM_2
 \|\nabla u(t)\|_{L^2}^2}{r_0^4}.$$
 Combining the above two cases we have
 $$\dfrac{\mathrm{d}}{\mathrm{d}t}A(t)\leq CM_2
 \|\nabla u(t)\|_{L^2}^2\max\left\{\dfrac{A(t)^{\frac{4}{3}}}{(\varepsilon K(\varepsilon))^{\frac{8}{3}}},\ \ \frac{1}{r_0^4}\right\}.$$
 Denote $F(y)=\dint_y^{+\infty}\left[\max\left\{\dfrac{y^{\frac{4}{3}}}
 {(\varepsilon K(\varepsilon))^{\frac{8}{3}}},\ \ \frac{1}{r_0^4}\right\}\right]^{-1}\mathrm{d}y,$ then
 $$
 \dfrac{\mathrm{d}}{\mathrm{d}t}F(A(t))\geq CM_2
 \|\nabla u(t)\|_{L^2}^2,
 $$
  and we can use the energy identity to obtain
 $$\ F(A(0))-F(A(t))\leq
 \int_0^t{CM_2\|\nabla u(s)\|_{L^2}^2}\mathrm{d}s\leq
 {C_1M_2\|u_0\|_{L^2}^2}.
 $$
 Therefore, if the condition $(a)':\ F( A(0))>C_1M_2\|u_0\|_{L^2}^2$ is satisfied, then $$\inf\limits_{0<t<T^*}F(A(t))>0,\  \sup\limits_{0<t<T^*}A(t)<+\infty,\  \sup\limits_{0<t<T^*}\|\Omega(t)\|_{L^2}^2<+\infty, $$ and these  imply the global regularity.

 Now we claim that, if $C_0>C_*\max\left\{1,\sqrt{{C_1}/{3}}\right\}$, then $\eqref{b}$ implies condition $(a)'$. Here $C_1,\ C_*$ are absolute positive constants. Notice that
 $$
 \ba{ll}F(A(0))&\geq F\left(\max\left\{A(0),\dfrac{(\varepsilon K(\varepsilon))^2}{r_0^3}\right\}\right)\\[4mm]&={3\max\left\{A(0),\dfrac{(\varepsilon K(\varepsilon))^2}{r_0^3}\right\}^{-\frac{1}{3}}(\varepsilon K(\varepsilon))^{\frac{8}{3}}},\ea
 $$
 $$A(0)^{\frac{1}{2}}=\left(\|J(0)\|_{L^2}^2+
 \dfrac{\varepsilon^{\frac{2}{3}}}{2}\|\Omega(0)\|_{L^2}^2\right)^{\frac{1}{2}}
 \leq\|J_0\|_{L^2}+\|\Omega_0\|_{L^2},
 $$ from the definition of $M_2,\ M_1,\ M_0$ and $(1.4)$ we have
 $$M_2<\varepsilon^{-\frac{2}{3}}(1+\|\Gamma\|_{L^\infty})^2,\ \ \ \ M_2\|u_0\|_{L^2}^2\leq\varepsilon^{-\frac{2}{3}}M_1^2,\ \ \
 A(0)^{\frac{1}{2}}M_1^3\leq M_0.
 $$
And
$\eqref{b}$ implies
$$K_0(\varepsilon)>C_0\max\left\{M_0^{1/4},r_0^{-1/2}{M_1}\right\},
$$
hence we obtain
 $$\ba{ll}\left(\dfrac{M_2\|u_0\|_{L^2}^2}{F( A(0))}\right)^{\frac{3}{2}}& \leq\dfrac{\varepsilon^{-1}M_1^3\max
 \left\{A(0)^{\frac{1}{2}},{\varepsilon K(\varepsilon)}/{r_0^{3/2}}\right\}}{3^{3/2}(\varepsilon K(\varepsilon))^4}\\[4mm] & \leq\dfrac{\max
 \left\{M_0,\varepsilon K(\varepsilon){M_1^3}/{r_0^{3/2}}\right\}}{3^{3/2}\varepsilon
 (\varepsilon K(\varepsilon))^4}\\[3mm]&
 \leq3^{-\frac{3}{2}}\max
 \left\{\dfrac{M_0C_*^4}{K_0(\varepsilon)^4},
 \dfrac{C_*^3M_1^3}{K_0(\varepsilon)^3r_0^{3/2}}\right\}\\[4mm] &<
 3^{-\frac{3}{2}}\max
 \left\{\dfrac{C_*^4}{C_0^4},
 \dfrac{C_*^3}{C_0^3}\right\}=3^{-\frac{3}{2}}
 \left(\dfrac{C_*}{C_0}\right)^3\\ &
 \leq3^{-\frac{3}{2}}
 \left(\sqrt{\dfrac{3}{C_1}}\right)^3=C_1^{-\frac{3}{2}},\ea
 $$
 $$\ \ \ \ \ \mbox{and} \ \ \ \  \ \ F( A(0))>C_1M_2\|u_0\|_{L^2}^2.$$
  Therefore the claim is true, this completes the proof of Theorem 1.1.\end{proof}

\begin{proof}[\bf Proof of\ Corollary\ \ref{cor1.1}.] First, we can take $r_0\in(0,\delta_0)$ such that
 $$r_0^{-1/2}{M_1}\geq M_0^{1/4},\ \ C_0M_1r_0^{-1/2}+1< e^{-1}r_0^{-1},$$ by $\eqref{1.6}$, we have $|\ln r_0|^{-3/2}\geq\|\Gamma\|_{L^\infty(r\leq r_0)}$, using the property of this $r_0$ we have $$\ba{c}\ \ \left(1+\ln\left(C_0\max\left\{M_0^{1/4},r_0^{-1/2}{M_1}\right\}
 +1\right)\right)^{-3/2}\\
 =\left(1+\ln\left(C_0r_0^{-1/2}{M_1}
 +1\right)\right)^{-3/2}>\left(1+\ln\left( e^{-1}r_0^{-1}\right)\right)^{-3/2}
 =|\ln r_0|^{-3/2},\ea$$
  Therefore condition (a) in Theorem\ 1.1 is satisfied, and we can use Theorem\ 1.1 to get the global regularity, this completes the proof of\ Corollary\ \ref{cor1.1}. \end{proof}

{\bf Acknowledgments}  The author would like to thank the
professors Gang Tian and Zhifei Zhang for some valuable  suggestions. \  \

\end{document}